\documentclass[12pt]{article}%
\usepackage{amsmath,amssymb,amsfonts,amscd,amsthm}
\usepackage{amsfonts}
\usepackage{amsmath}
\usepackage{graphicx}%
\usepackage{float}
\usepackage[margin=1.5cm]{caption}
\usepackage{subcaption}
\usepackage{epstopdf}
\usepackage{xcolor}
\usepackage{multirow}
\usepackage[utf8]{inputenc}

\usepackage{tikz}
\usetikzlibrary{arrows,shapes,positioning}
\usetikzlibrary{decorations.markings}
\tikzstyle arrowstyle=[scale=1]
\tikzstyle directed=[postaction={decorate,decoration={markings,
    mark=at position .65 with {\arrow[arrowstyle]{stealth}}}}]
\tikzstyle reverse directed=[postaction={decorate,decoration={markings,
    mark=at position .65 with {\arrowreversed[arrowstyle]{stealth};}}}]
    
\providecommand{\U}[1]{\protect\rule{.1in}{.1in}}
\newtheorem{theorem}{Theorem}

\newtheorem{prop}[theorem]{Proposition}

\newcommand{\z}{{\bf z}}
\newcommand{\Z}{{\bf Z}}

\newcommand{\suchthat}{\;\ifnum\currentgrouptype=16 \middle\fi|\;}

\begin{document}
  \begin{center}
        {\fontsize{18}{22}\selectfont
       \bf Regularization of the restricted $(n+1)$--body problem on curved spaces}
       \end{center}

\vspace{4mm}

        \begin{center}
        {\bf Ernesto P\'erez-Chavela$^1$, and Juan Manuel S\'anchez-Cerritos$^2$}\\
\bigskip
$^1$Departamento de Matem\'aticas\\
Instituto Tecnol\'ogico Aut\'onomo de M\'exico, Mexico City, Mexico\\
\bigskip
$^2$Department of Mathematics\\
Sichuan University, Chengdu, People's Republic of China\\
\bigskip
ernesto.perez@itam.mx, sanchezj01@gmail.com
       \end{center}

        %\vspace{1mm}
        
%        \begin{center}
%        \today
%        \end{center}

        \abstract{ We consider $(n+1)$ bodies moving under their mutual gravitational attraction in spaces with constant Gaussian curvature $\kappa$. In this system, $n$ primary bodies with equal masses form a relative equilibrium solution with a regular polygon configuration, and the remaining body of negligible mass does not affect the motion of the others. We show that the singularity due to binary collision between the negligible mass and the primaries can be regularized local and globally through suitable changes of coordinates (Levi-Civita and Birkhoff type transformations). }
        
        \
        
        
\section{Introduction}

We consider the generalization of the gravitational $n-$body problem to spaces of constant curvature proposed by Diacu, P\'erez-Chavela and Santoprete \cite{Diacu,Diacu2}. The problem has its roots on the ideas about non-Euclidean geometries proposed by Lovachevski and Bolyai in the 19th century \cite{Bolyai,Lovachevski}. For more details about the history of this fascinating problem we refer the interested readers to \cite{Diacu4}.

The restricted $(n+1)-$body problem  on curved spaces refers to the study of the motion of a particle $q$, with negligible mass, moving under the gravitational attraction of $n$ other particles, called primaries. The mass of the particle $q$ is very small, in such a way that the motion of the primaries are not affected by  this particle. In a personal notification, Carles Sim\'o pointed us that through a suitable change of coordinates follow by a rescaling of time, we can focus our analysis on the cases of curvature $\kappa = 1$ and  $\kappa = -1$ (see \cite{Diacu4} for details). In this way, 
we consider the  sphere embedded in $R^3$ with radius $1$, denoted by $\mathbb{S}^2$,  as model for positive curvature, and the  sphere (with Lorentzian metric) of imaginary radius $i$, denoted by 
$\mathbb{H}^2$, as model for negative curvature.

In the classical case (zero curvature), the restricted three-body problem was firstly proposed by Euler in the 18th century. One of the main obstacles to study the restricted problems is the presence of singularities due to collision. We start with a fix configuration for the primaries, avoiding the collision among them, but the problem of collision of the massless particle with one or more of the primaries still persists.
This problem was tacked first  by Levi-Civita \cite{Levi} and some years later by Birkhoff \cite{ Birkhoff}. The first technique is useful to regularize just one singularity, that we call local regularization. There are several generalizations of this technique which are applied to the restricted three body problem, see for instance \cite{Roman}, Levi-Civita regularization and its generalizations 
is very helpful for analyze the dynamics of a satellite in an orbit close to another massive body or in general in spatial missions to explore a single planet. However, sometimes it is necessary to have a global picture of the solutions, for instance to study escapes of particles or a possible connection 
among the equilibrium points, it is necessary to have a global regularization of all singularities due to collision, in this last case we use Birkhoff technique. Both techniques consist of suitable change of coordinates follow for a reparametrization of the time.

Some years ago, this classical problem (zero curvature) problem was extended considering more bodies , and the problem of regularize the binary collisions where also tackled \cite{ Vidal2}. Recently the restricted three body problem was also proposed considering the motion on curved spaces, for the restricted $2$--body problem,  where the two primaries  of different masses move on different planes around the $z$-axis \cite{Vidal3}.

In this paper we consider the restricted $(n+1)$--body problem on $\mathbb{S}^2$ and $\mathbb{H}^2$ where $n$ particles with equal masses are moving on a plane parallel to the $xy$-plane forming a regular polygon configuration. We find the regularization transformations  that allows avoid the singularities due to collisions between these particles and the remaining body of negligible mass. It results that the transformations are similar as in the classical Newtonian problem. We believe that these problems can have important applications in the dynamics of the components of an atom and in the new astronomical discoveries, usually studied only through quantum mechanics and relativity, nevertheless we let the analysis of the possible applications for a forthcoming paper and we will concentrate here just in the theoretical aspects of the problem.

After the introduction the paper is organized as follows. In section 2, we show the existence of polygonal relative equilibria given by the primary $n$ bodies with equal masses. In section 3 we set the problem and present the equations in a convenient form (a suitable rotating frame). In section 4 we present the main results of the work, in this way it is necessary to use the stereography projection to avoid the constraints and allow us to work with complex variables. Then we obtain the local and global regularization of the binary collisions.

\section{Relative Equilibria}

We consider the motion of the primary $n$ bodies where each particle, denoted by $q_i$, moves on the space $\mathbb{S}^2$ or $\mathbb{H}^2$. In this section we will show the existence of solutions given by relative equilibria.

The motion of the particles is lead by the cotangent potential

\[ U(q)=\sum_{i<j}m_im_j cotn(d(q_i,q_j)), \]
where $d$ is the geodesic distance on the corresponding space $\mathbb{S}^2$ or $\mathbb{H}^2$, and $cotn$ means the usual cotangent function or the hyperbolic cotangent function respectively.

The kinetic energy is defined as

\[ T=\dfrac{1}{2}\sum_{i}m_i\dot{q}_i \odot\dot{q}_i, \]
where the symbol $\odot$ represents the usual inner product if we consider $\mathbb{S}^2$, or the Lorentzian inner product if  $\mathbb{H}^2$ is considered (in this case for, $a,b \in \mathbb{R}^3$ we have 
$a \odot b = a_xb_x + a_yb_y - a_zb_z$).

From the Euler-Lagrange equations, the equations of motion take the form 

\begin{equation}\label{systemS2}
\ddot{q}_i=\sum_{j=1,j\neq i}^n\dfrac{q_j-\sigma(q_i\odot q_j)q_i}{[\sigma-\sigma(q_i\odot q_j)^2]^{3/2}}-\sigma(\dot{q}_i\odot \dot{q}_i)q_i, \ \ \ i=1,\cdots, n,
\end{equation}
where $\sigma$ stands for $1$ if we analyze $\mathbb{S}^2$ or $-1$ for $\mathbb{H}^2$.

\subsection{Relative Equilibria on $\mathbb{S}^2$}

Consider the group of isometries $SO(3)$ acting 
on $\mathbb{R}^3$. It is well known that it consists of all  uniform rotations. The relative equilibria are invariant solutions of the equations of motion under the group $SO(3)$. Now, since the principal axis theorem states that any $A \in SO(3)$ can be written, in some orthonormal  basis, as a rotation about a fixed axis, we consider this one as the $z$ axis. Hence we can characterize this result as follows

\begin{prop}
A solution $q_i, i=1,\cdots,n$ of the equations of motion on $\mathbb{S}^2$ is a relative equilibrium if and only if $q_i=(x_i,y_i,z_i)$, with $x_i=r_i  \cos(\Omega t + \alpha_i),$ $y_i=r_i  \sin(\Omega t + \alpha_i),$ and $z_i= \sqrt{1-r_i^2}$, where $\Omega, \alpha_i$ and $r_i, i=1,\cdots,n$ are constants.
\end{prop} 

\begin{proof}
The result follows directly by straightforward computations, we omit the details here.
\end{proof}

Now we will show that, if in the above proposition, $z_i=z_j, \,\, \forall i,j$,  then it is possible to find relative equilibria.

\begin{theorem}
For $n$ equal masses on $\mathbb{S}^2$ with a regular polygon initial configuration with the bodies at a height $z=$constant $\neq 0$, there exist a positive and a negative value for the initial velocity such that the solution is a relative equilibrium.
\end{theorem}

\begin{proof}
Consider $n$ particles with equal masses $m=1$ with a regular polygon initial configuration.

The position for the $i-$th body at a given time $t$ is $q_i(t)=(x_i(t),y_i(t),z(t))$ where

\[x_i(t)=r \cos\left[ \Omega t +(i-1)\dfrac{2 \pi}{n}\right], \ \ \ y_i(t)=r \sin\left[ \Omega t +(i-1)\dfrac{2 \pi}{n}\right], \ \ z(t) \neq 0.\]

We will show that there exist values of $\Omega$ such that the above functions satisfy (\ref{systemS2}).

By symmetry we can consider only the equations of motion for the $x_i(t)$ coordinates.

We have

\begin{equation}
\begin{split}
 x_{i+k}=&r \cos\left[ \Omega t +(i+k-1)\dfrac{2 \pi}{n}\right],\\
 x_{i-k}=&r \cos\left[ \Omega t +(i-k-1)\dfrac{2 \pi}{n}\right],\\
 \dot{x}_i(t)=&-\Omega r\sin\left[ \Omega t +(i-1)\dfrac{2 \pi}{n}\right],\\
 \ddot{x}_i(t)=&-\Omega^2 r\cos\left[ \Omega t +(i-1)\dfrac{2 \pi}{n}\right].
 \end{split}
 \end{equation}

Let $A=\Omega t+(i-1)\dfrac{2\pi}{n}$. We have for $n$ odd

\begin{equation}
\begin{split}
 \ddot{x}_i=&\sum_{j=1,j\neq i}^n\dfrac{x_j-(q_i\cdot q_j)x_i}{[1-(q_i\cdot q_j)^2]^{3/2}}-(\dot{q}_i\cdot \dot{q}_i)x_i.
 \end{split}
\end{equation}
For each $i$ we enumerate the particles as $(-\frac{n+1}{2}+i+1\cdots,i-1,i,i+1,\cdots, \frac{n+1}{2}+i-1)$. Hence
\begin{equation}
\begin{split}
 \ddot{x}_i
 %=&\dfrac{x_1-(q_i\cdot q_1)x_i}{[1-(q_i\cdot q_1)^2]^{3/2}}+\dfrac{x_{n-i+2}-(q_i\cdot q_{n-i+2})x_i}{[1-(q_i%\cdot q_{n-i+2})^2]^{3/2}}+2\sum_{j=i+1}^{\frac{n+1}{2}}\dfrac{x_j-(q_i\cdot q_j)x_i}{[1-(q_i\cdot q_j)^2]^{3/2}}+2\sum_{j=2}^{i-1}\dfrac{x_j-(q_i\cdot q_j)x_i}{[1-(q_i\cdot q_j)^2]^{3/2}}-(\dot{q}_i\cdot %\dot{q}_i)x_i\\
 =&\sum_{j=i+1}^{\frac{n+1}{2}+i-1}\dfrac{x_j-(q_i\cdot q_j)x_i}{[1-(q_i\cdot q_j)^2]^{3/2}}+\sum_{j=i-1}^{-\frac{n+1}{2}+i+1}\dfrac{x_j-(q_i\cdot q_j)x_i}{[1-(q_i\cdot q_j)^2]^{3/2}}-(\dot{q}_i\cdot \dot{q}_i)x_i\\	
 =&\sum_{j=1}^{\frac{n+1}{2}-1}\dfrac{x_{j+i}-(q_i\cdot q_{j+i})x_i}{[1-(q_i\cdot q_{j+i})^2]^{3/2}}+\sum_{j=-1}^{-\frac{n+1}{2}+1}\dfrac{x_{j+i}-(q_i\cdot q_{j+i})x_i}{[1-(q_i\cdot q_{j+i})^2]^{3/2}}-(\dot{q}_i\cdot \dot{q}_i)x_i\\
 =&\sum_{j=1}^{\frac{n+1}{2}-1}\dfrac{x_{j+i}-(q_i\cdot q_{j+i})x_i}{[1-(q_i\cdot q_{j+i})^2]^{3/2}}+\sum_{j=1}^{\frac{n+1}{2}-1}\dfrac{x_{i-j}-(q_i\cdot q_{i-j})x_i}{[1-(q_i\cdot q_{i-j})^2]^{3/2}}-(\dot{q}_i\cdot \dot{q}_i)x_i.\\
 \end{split}
\end{equation}

Notice that $\ddot{x}_i=-\Omega^2 x_i$, \  $x_{i\pm j}=x_i\cos(2j\pi /n)\mp y_i\sin(2j\pi /n)$, \  $q_i\cdot q_{i-j}=q_i\cdot q_{i+j}=r^2 \cos(2j\pi/n)+1-r^2$ \  and \  $\dot{q}_i\cdot \dot{q}_i=r^2 \Omega^2$. With these facts we obtain

\begin{equation}\label{w1}
\Omega^2=2\sum_{j=1}^{\frac{n+1}{2}-1}\dfrac{1-\cos(2j\pi/n)}{[1-(r^2 \cos(2j\pi/n)+1-r^2)^2]^{3/2}}.
\end{equation}

For $n$ even we have

\begin{equation}
\begin{split}
 \ddot{x}_i=&\sum_{j=1,j\neq i}^n\dfrac{x_j-(q_i\cdot q_j)x_i}{[1-(q_i\cdot q_j)^2]^{3/2}}-(\dot{q}_i\cdot \dot{q}_i)x_i.
 \end{split}
\end{equation}
For each $i$ we enumerate the particles as $(-\frac{n}{2}+i+1\cdots,i-1,i,i+1,\cdots, \frac{n}{2}+i-1)$. Hence
\begin{equation}
\begin{split}
 \ddot{x}_i=&\sum_{j=i+1}^{\frac{n}{2}+i-1}\dfrac{x_j-(q_i\cdot q_j)x_i}{[1-(q_i\cdot q_j)^2]^{3/2}}+\sum_{j=i-1}^{-\frac{n}{2}+i+1}\dfrac{x_j-(q_i\cdot q_j)x_i}{[1-(q_i\cdot q_j)^2]^{3/2}}\\
 & +\dfrac{x_{\frac{n}{2}+i+1}-(q_i\cdot q_{\frac{n}{2}+i+1})x_i}{[1-(q_i\cdot q_{\frac{n}{2}+i+1})^2]^{3/2}}-(\dot{q}_i\cdot \dot{q}_i)x_i\\	
 =&\sum_{j=1}^{\frac{n}{2}-1}\dfrac{x_{j+i}-(q_i\cdot q_{j+i})x_i}{[1-(q_i\cdot q_{j+i})^2]^{3/2}}+\sum_{j=1}^{\frac{n}{2}-1}\dfrac{x_{i-j}-(q_i\cdot q_{i-j})x_i}{[1-(q_i\cdot q_{i-j})^2]^{3/2}}-(\dot{q}_i\cdot \dot{q}_i)x_i\\
&+ \dfrac{-x_{i}-(q_i\cdot q_{\frac{n}{2}+i+1})x_i}{[1-(q_i\cdot q_{\frac{n}{2}+i+1})^2]^{3/2}}-(\dot{q}_i\cdot \dot{q}_i)x_i.
\end{split}
\end{equation}

And we have

\begin{equation}\label{w2}
\Omega^2=2\sum_{j=1}^{\frac{n}{2}-1}\dfrac{1-\cos(2j\pi/n)}{[1-(r^2 \cos(2j\pi/n)+1-r^2)^2]^{3/2}}+\dfrac{1}{4r^3(1-r^2)^{3/2}}.
\end{equation}

The equations (\ref{w1}) and (\ref{w2}) are positive, hence we conclude that there exist a positive and a negative value of $\Omega$ such that generate relative equilibria.

\end{proof}

If the particles are on the equator of $\mathbb{S}^2$, then we have the following \cite{paper1}:  if $n$ is odd, then and they have equal masses forming a regular polygon configuration moving with constant angular velocity $\Omega$, then the positions and velocities form a solution of relative equilibrium for any $\Omega \in \mathbb{R}$.

\subsection{Relative Equilibria on $\mathbb{H}^2$}

Consider the orthogonal transformations of determinant $\pm 1$ that leave $\mathbb{H}^2$ invariant, this is a closed group called \textit{Lorentz group}, $Lor(\mathbb{R}^{2,1},\odot)$. The principal axis theorem in this case states that every $G \in Lor(\mathbb{R}^{2,1},\odot)$ has one of the following canonical forms:

\[A=P  \left( \begin{array}{ccc}
\cos \theta & -\sin \theta & 0 \\
\sin \theta & \cos \theta & 0 \\
0 & 0 & 1 \end{array} \right) P^{-1},\]

\[B=P  \left( \begin{array}{ccc}
1 & 0 & 0 \\
0 & \cosh s & \sinh s \\
0 & \sinh s & \cosh s \end{array} \right) P^{-1},\]
or

\[C= P  \left( \begin{array}{ccc}
1 & -t & t \\
t & 1-\frac{t^2}{2} & \frac{t^2}{2} \\
t & -\frac{t^2}{2} & 1+\frac{t^2}{2} \end{array} \right) P^{-1},\]
where $\theta \in [0,2 \pi),s,t \in \mathbb{R}$, and $P \in Lor(\mathbb{R}^{2,1},\odot)$.

The above transformations are called elliptic, hyperbolic, and parabolic, respectively. %We will call relative equilibria to those solutions generated by the above transformations: elliptic relative equilibria, hyperbolic relative equilibria or parabolic relative equilibria.

It is well known from people in the field that there are not solutions generated by parabolic transformations \cite{Diacu}. Those solutions generated by elliptic or hyperbolic transformations are called elliptic or hyperbolic relative equilibria. 

In this work we are interested in solutions where the $n$ primaries form relative equilibria with a regular polygon configuration. In a recent paper of the authors, they proved the no existence of these kind of hyperbolic relative equilibria, in particular the no existence of Lagrange hyperbolic relative equilibria \cite{paperh2}. Hence we will focus only on elliptic relative equilibria. 
Analogously as in $\mathbb{S}^2$ we have the following result.

\begin{theorem}
For $n$ equal masses on $\mathbb{H}^2$ with a regular polygon initial configuration with the bodies at a height $z=$constant $\neq 0$, there exist a positive and a negative value for the initial velocity such that the solution is a relative equilibrium.
\end{theorem}

\begin{proof}
The proof is similar by noticing that $q_i \odot q_{i-j}=q_i \odot q_{j-i}=r^2\cos(\frac{2\pi j}{n})-1-r^2$. 

For $n$ odd the angular velocity satisfies

\begin{equation}\label{w3}
\Omega^2=2\sum_{j=1}^{\frac{n+1}{2}-1}\dfrac{1-\cos(2j\pi/n)}{[-1+(r^2 \cos(2j\pi/n)-1-r^2)^2]^{3/2}}.
\end{equation}

For $n$ even,

\begin{equation}\label{w4}
\Omega^2=2\sum_{j=1}^{\frac{n}{2}-1}\dfrac{1-\cos(2j\pi/n)}{[-1+(r^2 \cos(2j\pi/n)-1-r^2)^2]^{3/2}}+\dfrac{1}{4r^3(1+r^2)^{3/2}}.
\end{equation}

Since equations (\ref{w3}) and (\ref{w4}) are positive we conclude that there exists a positive and negative value for the angular velocity that leads to elliptic relative equilibria.

\end{proof}

\section{The restricted curved $(n+1)$--body problem}

In order to unified our analysis we introduce the notation $\mathbb{M}^2$ without distinction for 
$\mathbb{S}^2$ or $\mathbb{H}^2$. The restricted curved $(n+1)-$body problem refers to the study of a system of $n+1$ particles moving on under their mutual attraction  on  $\mathbb{M}^2$, where $n$ bodies with positions $q_i$ of equal masses (called primary bodies) are rotating on a circle parallel to the $xy$ plane with  velocity (\ref{w1}) or (\ref{w2}) and with a regular polygon configuration. The remaining body located at position $q$ has a negligible mass and its motion is given by the following equation

\begin{equation}
\ddot{q}=\sum_{i=1}^n\dfrac{q_i-\sigma(q_i\odot q)q}{[\sigma-\sigma(q\odot q_i)^2]^{3/2}}-\sigma(\dot{q} \odot \dot{q})q.
\end{equation}

As in the classical case, we introduce rotating coordinates. Let $q=RQ$ with $Q=(\xi,\eta,\vartheta)^T$, where $R$ is the rotation matrix 

\[R= \left( \begin{array}{ccc}
\cos \Omega & -\sin \Omega & 0 \\
\sin \Omega & \cos \Omega & 0 \\
0 & 0 & 1 \end{array} \right).\]

After a straightforward computation the new equation of motion in the new variables are

\begin{equation}\label{eqn}
\begin{split}
 \ddot{Q}-2\Omega J \dot{Q}+ [\sigma(\dot{\xi}-\Omega \eta)^2+\sigma(\dot{\eta}+\Omega \xi)^2+\dot{\vartheta}^2]Q =\nabla_{Q}\left(\dfrac{\Omega^2}{2}(\xi^2+\eta^2)\right. \\
 \left. +\sum_{i=1}^n\dfrac{Q_i\odot Q}{\left(\sigma-\sigma(Q_i \odot Q)^2\right)^{1/2}}\right),
 \end{split}
\end{equation}

with

\[J= \left( \begin{array}{ccc}
0 & 1 & 0 \\
-1& 0 & 0 \\
0 & 0 & 0 \end{array} \right).  \]

In these coordinates the position of the primaries take the form

\[ Q_i=\left[r\cos \left(\frac{2 \pi}{n}(i-1)\right),r\sin \left(\frac{2 \pi}{n}(i-1)\right),z\right]. \]

\bigskip

\textbf{Stereographic projection }

In our analysis we consider the stereographic projection from the point $(0,0,-1)$ to $\mathbb{R}^2$, 
\ $\Pi: \mathbb{M}^2  \rightarrow \mathbb{R}^2$. This function maps $Q\longmapsto (u,v)$, with

\[ u=\dfrac{\xi}{1+\vartheta}, \ \ v=\dfrac{\nu}{1+\vartheta}. \]

The inverse function $\Pi^{-1}$ maps $(u,v) \longmapsto Q$ where

\[ \xi=\dfrac{2u}{1+\sigma(u^2+v^2)}, \ \ \eta=\dfrac{2u}{1+\sigma(u^2+v^2)}, \ \ \vartheta=\dfrac{1-\sigma(u^2+v^2)}{1+\sigma(u^2+v^2)}. \]

It is known that $\Pi$ maps $\mathbb{S}^2$ onto the whole plane $\mathbb{R}^2$, with the metric $ds^2=\dfrac{4}{1+u^2+v^2}$, this plane with this metric is known by some authors as the curved plane. The case of $\mathbb{H}^2$, this space is projected onto the open unitary disk with the metric $ds^2=\dfrac{4}{1-u^2-v^2}$, this space is the well known model of hyperbolic geometry called the Poincar\'e disk.

Under $\Pi$, the primaries, originally on $\mathbb{M}^2$, now are locate  at $w_i=\dfrac{1}{1+z}(k_i,h_i)$,  with $k_i=r\cos \left(\frac{2 \pi}{n}(i-1)\right)$ and $h_i=r\sin \left(\frac{2 \pi}{n}(i-1)\right)$.

In the right part of (\ref{eqn}), the so called effective potential, can be written as

\begin{equation}
\begin{split}
 &\dfrac{\Omega^2}{2}(\xi^2+\eta^2)+\sum_{i=1}^n\dfrac{Q_i\odot Q}{\left(\sigma-\sigma(Q_i \odot Q)^2\right)^{1/2}}\\
 =& \dfrac{2 \Omega^2(u^2+v^2)}{(1+\sigma(u^2+v^2))}+\sum_{i=1}^n\dfrac{2k_iu+2h_iv+\sigma z(1-\sigma(u^2+v^2))}{[\sigma(1+\sigma(u^2+v^2))^2-\sigma(2k_iu+2h_iv+\sigma z(1-\sigma(u^2+v^2)))^2]^{1/2}}\\
 =:&\Psi(u,v)+U(u,v).
 \end{split}
\end{equation}

\section{Regularization}

We first write the problem as a Hamiltonian system, with Hamiltonian function given by

\begin{equation}\label{ham}
 H(u,v,p_u,p_v)=\dfrac{(1+\sigma(u^2+v^2))^2}{8}(p_u^2+p_v^2)+\Omega(vp_u-up_v)-U(u,v),
\end{equation}
where $U$ is the corresponding potential getting from the stereographic projection of $\mathbb{S}^2$ or $\mathbb{H}^2$ onto $\mathbb{R}^2$. If no confusion arises, we will keep denoting the positions of the primaries on $\mathbb{C}$ as $w_i$.

In order to analyze the regularization of the binary collisions between the negligible mass with the primaries, we consider  complex variables through the following change of coordinates

\[ \z=u+iv, \ \ \Z=p_u+ip_v.\]

Then (\ref{ham}) takes the form

\begin{equation} \label{H}
H=\dfrac{(1+\sigma|\z|^2)2}{4}|\Z|+2\Omega Im(\z\overline{\Z})-2V(\z,\overline{\z}), 
\end{equation}
with

\[ V(\z,\bar{\z})=\sum_{j=1}^n\dfrac{k_j(\z+\overline{\z})-ih_j(\z-\overline{\z})+\sigma z(1-\sigma |\z|^2)}{r|\z-w_j||\z-\widehat{w}_j|},\]

Where $\hat{w}_i=-\dfrac{1}{1 -z}(k_i,h_i)$. 

If we consider $\mathbb{S}^2$ , then $\Pi^{-1}(\hat{w}_i)$  corresponds to the antipodal point of the primary $Q_i$, however if we consider  $\mathbb{H}^2$, then $\hat{w}_i$ is a point such that $\Pi^{-1}(\hat{w}_i)$ does not belong to $\mathbb{H}^2$.

\subsection{Local Regularization}

In this section we state the first main theorem of this paper.

\begin{theorem}
The transformation $\z=g(w)=w\overline{w}+w_k$ , $(i=1,\cdots n)$ with the time transformation $\dfrac{dt}{ds}=|w|^2$, regularize the singularity of (\ref{H}) due collision between the negligible mass and the $k-$th primary body.
\end{theorem}

\begin{proof}

First, let us consider the space transformation $\z=g(w)$ with $\Z=W/\overline{g'(w)}$, and the new time $s$ such that $\dfrac{dt}{ds}=|g'(w)|^2.$%\dfrac{(n-1)^2}{n^2}\dfrac{|z-w_1|^2|z-w_2|^2\cdots |z-w_n|^2}{|z|^{2n}}$.

Take a constant energy level $H=-\dfrac{C}{2}$, and let us define a new Hamiltonian $\hat{H}=|g'(w)|^2\left(H+\dfrac{C}{2}\right)$. The flow generated by both of the Hamiltonian functions is equivalent, hence we will consider the flow generated by $\hat{H}$ at zero energy level.

% It is known that the flow of the Hamiltonian system associated to $H$ at the level of energy $-C/2$ is the same as the flow of the Hamiltonian system related to $\hat{H}$ at a level of energy equals to  zero, with the corresponding time $s$. Hence we will be working with the Hamiltonian $\hat{H}$ with zero energy level.

Performing the change or variables  we obtain that the Hamiltonian takes the form

\begin{eqnarray}\label{H2}
\hat{H} &=& \dfrac{(1+\sigma |g(w)|^2)2}{4}|W| \ + \ 2|g'(w)|^2\Omega Im\left(g(w)\overline{W}/g'(w)\right) \nonumber \\
&&  - \ 2|g'(w)|^2V(w,\overline{w}) + |g'(w)|^2 \dfrac{C}{2}.
\end{eqnarray}

Take $z=g(w)=w\overline{w}+w_k$, then $g'(w)=\overline{w}$ and $|g'(w)|^2=|\overline{w}|^2=|w|^2=w\overline{w}=|w\overline{w}|$. We will check that this transformation avoids the singularity due collision of the negligible mass and the primary $k$.

On $\mathbb{M}^2$:
{\footnotesize
\begin{equation}
\begin{split}
|g'(w)|^2V(w,\overline{w})=&|w|^2\left[\sum_{j=1}^n\dfrac{\left(  k_j(g(w)+\overline{g(w)})-ih_j(g(w)-\overline{g(w)})+\sigma z(1-\sigma|g(w)|^2)\right)}{r|w\overline{w}+w_k-w_j||w\overline{w}+w_k-\widehat{w}_j|}\right]\\
=&|w|^2\left[\sum_{j=1,j\neq k}^n\dfrac{\left(  k_j(g(w)+\overline{g(w)})-ih_j(g(w)-\overline{g(w)})+\sigma z(1-\sigma|g(w)|^2)\right)}{r|w\overline{w}+w_k-w_j||w\overline{w}+w_k-\widehat{w}_j|}\right]\\
&+|w|^2\dfrac{\left(  k_j(g(w)+\overline{g(w)})-ih_j(g(w)-\overline{g(w)})+\sigma z(1-\sigma|g(w)|^2)\right)}{r|w\overline{w}||w^2+w_i-\widehat{w}_j|}\\
=&|w|^2\left[\sum_{j=1,j\neq k}^n\dfrac{\left(  k_j(g(w)+\overline{g(w)})-ih_j(g(w)-\overline{g(w)})+\sigma z(1-\sigma|g(w)|^2)\right)}{r|w\overline{w}+w_k-w_j||w\overline{w}+w_k-\widehat{w}_j|}\right]\\
&+\dfrac{\left(  k_j(g(w)+\overline{g(w)})-ih_j(g(w)-\overline{g(w)})+\sigma z(1-\sigma|g(w)|^2)\right)}{r|w^2+w_i-\widehat{w}_j|}.
\end{split}
\end{equation}
}

We can see that the singularities of equation (\ref{H2}) are avoided, hence we conclude the proof.

\end{proof}

\subsection{Global Regularization}

The second main statement of this work is the following.

\begin{theorem}
The transformation $\z=g(w)=\dfrac{n-1}{n}w+\dfrac{w_1^n}{nw^{n-1}}$ and the time transformation $\dfrac{dt}{ds}=\dfrac{(n-1)^2}{n^2}\dfrac{|\z-w_1|^2|\z-w_2|^2\cdots |\z-w_n|^2}{|\z|^{2n}}$ regularize the $n$ binary collision-singularities of (\ref{H}), between the negligible mass and each of the primaries.
\end{theorem}

\begin{proof}
Consider the transformation $\z=g(w):=\alpha w +\dfrac{\beta}{w^{n-1}}$, $\Z=W/\overline{g'(w)}$, and the time $s$ such that $\dfrac{dt}{ds}=|g'(w)|^2.$%\dfrac{(n-1)^2}{n^2}\dfrac{|z-w_1|^2|z-w_2|^2\cdots |z-w_n|^2}{|z|^{2n}}$.

As we did before, we will consider a fixed energy level $-\dfrac{C}{2}$ and a new Hamiltonian defined as $\hat{H}=|g'(w)|^2\left(H+\dfrac{C}{2}\right)$. We will be working using this new Hamiltonian at zero energy evel.

Performing the change or variables  we obtain that the Hamiltonian takes the form

\begin{eqnarray}\label{H1}
\hat{H} &=& \dfrac{(1+\sigma|g(w)|^2)2}{4}|W| \ + \ 2|g'(w)|^2\Omega Im\left(g(w)\overline{W}/g'(w)\right) \nonumber \\
 && -  2|g'(w)|^2V(w,\overline{w}) + |g'(w)|^2 \dfrac{C}{2}.
\end{eqnarray}

In order to remove singularities we will find $\alpha$ and $\beta$ with the following properties: First that the primaries remain fixed, it means

$$\text{On} \ \ S^2: \ \ \ g(w_i)=w_i, \ \ g(\hat{w}_i)=\hat{w}_i; \ \ \  \text{On} \ \ H^2: g(w_i)=w_i,$$

and second that the functions $g(w)$ allows us to remove all the collision singularities, it means $g'(w_i)=0$, or
 
 $g'(w)=\alpha+\dfrac{\beta (n-1)}{w^n}=\dfrac{\alpha}{w^n}\left( w^n-\dfrac{\beta}{\alpha} (n-1) \right)=\dfrac{\alpha}{w^n}(w-w_1)(w-w_2)\cdots (w-w_n)$.
 
 Last equality is satisfied if the following occurs
 
 \begin{equation}
 \begin{split}
 0=&\sum_{j=1}^nw_j,\\
 0=&w_1\sum_{j=2}^nw_j+w_2\sum_{j=3}^nw_j+\cdots +w_{n-2}\sum_{j=n-1}^nw_j+w_{n-1}w_n,\\
 0=&w_1w_2\sum_{j=3}^nw_j+w_1w_3\sum_{j=4}^nw_j+ \cdots  +w_2w_3\sum_{j=4}^nw_j +\cdots +w_2w_4\sum_{j=5}^nw_j+\\
&+w_{n-3}w_{n-2}\sum_{j=n-1}^nw_j +w_{n-2}w_{n-1}w_{n}=0,\\
\vdots\\
0=&\sum_{i=1}^n \prod_{j=1, j\neq i}^nw_j,\\
(-1)^n&\prod_{j=1}^nw_j=-\dfrac{\beta}{\alpha}(n-1).
 \end{split}
 \end{equation}
 
 The first $(n-1)$ conditions are satisfied if the first one occurs, and it is true since it is the sum of the $n$th roots of the unity.
 
 Since $w_{j+1}=e^{i2\pi j/n}w_1$, we have $\dfrac{w_1^n}{n-1}=\dfrac{\beta}{\alpha}$. We also have that the property $g(w_1)=w_1$ implies $\dfrac{\beta}{\alpha}=\dfrac{w_1^n}{n-1}$. With these facts we conclude $\alpha=\dfrac{n-1}{n}$ and $\beta=\dfrac{w_1^n}{n}$.

Notice that

\[ g(w)-w_i=\dfrac{n-1}{n}w+\dfrac{w_1^n}{nw^{n-1}}-w_i=\dfrac{(w-w_i)^2}{w^{n-1}}G(w), \]
where

\[ G(w)=\sum_{k=0}^{n-2}\dfrac{n-k-1}{n}w^{n-2-k}w_i^{k}. \]

And we have that $G(w_i)\neq 0$, for $i=1,\cdots, n$.

We now check  that the singularities of (\ref{H1}) due collisions are removed.

{\footnotesize
\begin{equation}
\begin{split}
|g'(w)|^2V(w,,\overline{w})=&\left[\sum_{j=1}^n\dfrac{\left(  k_j(g(w)+\overline{g(w)})-ih_j(g(w)-\overline{g(w)})+\sigma z(1-\sigma|g(w)|^2)\right)}{r|g(w)-w_j||g(w)-\widehat{w}_j|}\right] \cdot\\
&\left[ \dfrac{(n-1)^2}{n^2w^{2n}}(w-w_1)^2\cdots (w-w_n)^2  \right]\\
=&\left[\sum_{j=1}^n\dfrac{w^{n-1}\left(  k_j(g(w)+\overline{g(w)})-ih_j(g(w)-\overline{g(w)})+\sigma z(1-\sigma|g(w)|^2)\right)}{r(w-w_j)^2G(w)|g(w)-\widehat{w}_j|}\right] \cdot\\
&\left[ \dfrac{(n-1)^2}{n^2w^{2n}}(w-w_1)^2\cdots (w-w_n)^2  \right]\\
=&\sum_{j=1}^n\left[ \dfrac{\left(  k_j(g(w)+\overline{g(w)})-ih_j(g(w)-\overline{g(w)})+\sigma z(1-\sigma|g(w)|^2)\right)}{rG(w)|g(w)-\widehat{w}_j|}\right. \cdot\\
&\left. \dfrac{(n-1)^2}{n^2w^{n+1}}(w-w_1)^2\cdots (w-w_{j-1})^2(w-w_{j+1})^2\cdots (w-w_n)^2  \right].
\end{split}
\end{equation}
}

Notice that the only singularity is the point $w=0$ which corresponds to $|\z|\rightarrow \infty$. With this we finish the proof.

\end{proof}

\subsection*{Acknowledgements} The first author has been partially supported by {\it Asociaci\'on Mexicana de Cultura A.C.}  The second author was supported by {\it The 2017's Plan of Foreign Cultural and Educational Experts Recruitment for the Universities Under the Direct Supervision of the Ministry of Education of China} (Grant no. WQ2017SCDX045).


\begin{thebibliography}{}


\bibitem{Vidal2} \'Alvarez Martha, Vidal Claudio: Global Regularization for the $(N+1)$--Body problem with the primaries in a regular $n$--gon central configuration, \emph{Qualitative Theory of Dynamical Systems} {\bf 14}-2, 175-187, 2015.

\bibitem{JA} Andrade Jaime, D\'avila Nestor, P\'erez-Chavela Ernesto, Vidal Claudio: Dynamics and regularization of the Kepler problem on surfaces of constant curvature, \emph{Canadian Journal of Mathematics} {\bf 30}, 1607-1626, (2017).
%\bibitem{JA} J. Andrade, N. D\'avila, E. Perez-Chavela and C. Vidal
%\emph{Dynamics and regularization of the Kepler problem on surfaces of constant curvature}, Can. J. Math.,  {\bf 69}-5,  961-991, (2017). \\
%http://dx.doi.org/10.4153/CJM-2016-014-5 , DOI 10.4153/CJM-2016-014-5. 

\bibitem{Vidal3} Andrade Jaime, P\'erez-Chavela Ernesto, Vidal Claudio: Regularization of the circular restricted three body problem on surfaces of constant curvature, \emph{Journal of Dynamics and Differential Equations} {\bf 30}, 1607-1626, (2017).
%\bibitem{Vidal3} J. Andrade, E. Perez-Chavela and C. Vidal, \emph{Regularization of the circular restricted three body problem on surfaces of constant curvature}, J. Dyn. Diff. Equat., {\bf 30}, 1607-1626, (2017).


\bibitem{Birkhoff} Birkhoff George: The restricted problem of three bodies, \emph{Rendiconti del Circolo Matematico di Palermo} {\bf 39}-1, (1915).
%\bibitem{Birkhoff} G.D. Birkhoff, \emph{The restricted problem of three bodies}. Rend Circ. Mat Palermo {\bf 39}-1, (1915).

\bibitem{Bolyai} Bolyai Wolfgang, Bolyai Johann: Geometrische Untersuchungen, Hrsg. P. St\"ackel, Teubner, Leipzig-Berlin, 1913.
%\bibitem{Bolyai} J. Bolyai and W.Bolyai, \emph{Geometrische Untersuchungen}. Hrsg. P. St\"ackel, Teubner, Leipzig-Berlin, 1913.


\bibitem{Diacu4} Diacu Florin: Relative equilibria of the Curved N-Body Problem, Atlantis Press, Series Volume 1, 2012.
%\bibitem{Diacu4} F. Diacu,
%\emph{Relative Equilibria of the Curved N-Body Problem}, Atlantis Press, Series Volume 1, 2012.


\bibitem{Diacu5} Diacu Florin: Polygonal homographic  orbits of the curved $n-$body problem, \emph{Transactions of the American Mathematical Society} {\bf 364}-5, 2783-2802, (2012).
%\bibitem{Diacu5} F. Diacu,
%\emph{Polygonal Homographic  Orbits of the Curved $n-$Body Problem}, Trans. Amer. Math. Soc.
% {\bf 364}-5, 2783-2802, (2012).
 

\bibitem{Diacu3} Diacu Florin: Homographic solutions of the curved 3-body problem, \emph{Journal of Differential Equations} {\bf 250},   340-366, (2011).
%\bibitem{Diacu3} F. Diacu and E. P\'erez-Chavela,
%\emph{Homographic solutions of the curved 3-body problem}, Journal of Differential Equations {\bf 250},   340-366, (2011).

\bibitem{Diacu} Diacu Florin: P\'erez-Chavela Ernesto, Santoprete Manuele, The $n$-body problem in spaces of constant curvature. Part I: Relative equilibria,  \emph{Journal of Nonlinear Science} {\bf 22},  247-266, (2012).
%\bibitem{Diacu} F. Diacu, E. P\'erez-Chavela and M. Santoprete,
%\emph{The $n$-Body Problem in Spaces of Constant Curvature.
%Part I: Relative Equilibria}, J. Nonlinear Sci. {\bf 22},  247-266, (2012).

\bibitem{Diacu2} Diacu Florin: P\'erez-Chavela Ernesto, Santoprete Manuele, The $n$-body problem in spaces of constant curvature. Part II: Singularities, \emph{Journal of Nonlinear Science} {\bf 22},  267-275, (2012).
%\bibitem{Diacu2} F. Diacu, E. P\'erez-Chavela and M. Santoprete,
%\emph{The $n$-Body Problem in Spaces of Constant Curvature.
%Part II: Singularities}, J. Nonlinear Sci. {\bf 22},  267-275, (2012).

\bibitem{paper1} Diacu Florin S\'anchez-Cerritos Juan Manuel, Zhu Shuqiang: Stability of fixed points and associated relative equilibria of the 3-body problem on $S^1$ and $S^2$, \emph{Journal of Dynamics and Differential Equations} {\bf 30}, 209-225, (2018).
%\bibitem{paper1} F. Diacu, J.M. S\'anchez-Cerritos and S. Zhu,
%\emph{Stability of Fixed Points and Associated Relative Equilibria of the 3-body Problem on $S^1$ and $S^2$}, Journal of Dynamics and Differential Equations, {\bf 30}, 209-225, (2018).


\bibitem{Lovachevski}  Lovachevski Nicolai: The new foundations of geometry with full theory of parallels. In Russian. 1835-1838, in Collected Works, V. 2, GITTL, Moscow, 159, 1949.
%\bibitem{Lovachevski} N. I. Lovachevski,
% \emph{The new foundations of geometry with full theory of parallels.}[In Russian]. 1835-1838, in Collected Works, V. 2, GITTL, Moscow, 159, 1949.
 

\bibitem{EPC} P\'erez-Chavela Ernesto, Reyes-Victoria Jos\'e Guadalupe: An intrinsec approach in the curved $n$-body problem. The positive curvature case, \emph{Transactions of the American Mathematical Society} {\bf 364}-7,  3805-3827, (2012).
%\bibitem{EPC} E. P\'erez-Chavela and J. G. Reyes-Victoria,
% \emph{An intrinsec approach in the curved $n$-body problem. The positive curvature case}, Trans. Amer. Math. Soc. {\bf 364}-7,  3805-3827, (2012).
 
 \bibitem{paper} P\'erez-Chavela Ernesto, S\'anchez-Cerritos Juan Manuel: Euler-type relative equilibria in spaces of constant curvature and their stability, \emph{Canadian Journal of Mathematics} {\bf 70}-2, 426-450, (2018).
%\bibitem{paper} E. P\'erez-Chavela and J. M. S\'anchez-Cerritos, \emph{Euler-type relative equilibria in spaces of constant curvature and their stability}, Canad. J.  Math. {\bf 70}-2, 426-450, (2018).
 
 \bibitem{paperh2} P\'erez-Chavela Ernesto, S\'anchez-Cerritos Juan Manuel: Hyperbolic relative equilibria for the negative curved $n-$body problem, \emph{Communications in Nonlinear Science and Numerical Simulation} https://doi.org/10.1016/j.cnsns.2018.07.022, (2018).
%\bibitem{paperh2} E. P\'erez-Chavela and J. M. S\'anchez-Cerritos, \emph{Hyperbolic relative equilibria for the negative curved $n-$body problem}, Communications in Nonlinear Science and Numerical Simulation. https://doi.org/10.1016/j.cnsns.2018.07.022, (2018).

\bibitem{Roman}  Roman R. Szucs-Csillik: Generalization of Levi-Civita regularization in the restricted three-body problem, \emph{Astrophysics and Space Science} {\bf 349},  117-123, (2014).
%\bibitem{Roman} R. Roman, I. Szucs-Csillik,
%\emph{Generalization of Levi-Civita regularization in the restricted three-body problem}, 
% Astrophys Space Sci. {\bf 349},  117-123, (2014).

  
  \bibitem{Levi} Szebehely Victor: Theory of orbits: The restricted problem of three bodies, Academic press, New York, 1967.
%\bibitem{Levi} V. Szebehely, \emph{Theory of orbits: The restricted problem of three bodies}, Academic press, New York, 1967.
 

\end{thebibliography}
\end{document}